\documentclass[12pt]{amsproc}

\usepackage{fancyhdr}

\author[Güloğlu]{İSMAİL Ş. GÜLOĞLU}
\address{Department of Mathematics, Doğuş University, Istanbul, Turkey}
\author[Ercan]{GÜLİN  ERCAN*}
\address{Department of Mathematics, Middle East Technical University, 06800, Ankara/Turkey}
\email{iguloglu@dogus.edu.tr}
\email{ercan@metu.edu.tr}
\thanks{$^{*}$Corresponding author}	
\newtheorem{theo}{\sc Theorem}[section]
\newtheorem{lem}[theo]{\sc Lemma}
\newtheorem{proposition}[theo]{\sc Proposition}

\newtheorem{definition}[theo]{\sc Definition}


\title{Fixing Size and Fitting height}

\begin{document}

\maketitle

\begin{abstract}

Let $G$ be a finite solvable group on which a nilpotent group $A$ acts by automorphisms. The fixing size $\mathbf{c}(G;A)$ of $A$ on $G$ is the number of $A$-composition factors on which $A$ acts trivially in an $A$-composition series of $G$. In this paper we obtain a linear bound for the Fitting height of $G$ in terms of $\mathbf{c}(G;A)$ and $\ell(A)$ where $\ell(A)$ denotes the number of prime divisors (counted with multiplicities) of $A$, under some additional hypotheses.   
\end{abstract}

\maketitle
\section{Introduction}

Throughout this paper, all groups are finite. The current study presents a modest contribution to the investigations dedicated to establishing the following well-known conjecture.

\textbf{Conjecture.} \textit{Let $A$ be a group acting fixed point freely on the group $G$.
The Fitting height $h(G)$ of $G$ is bounded above by $\ell(A)$,that is, the number of prime divisors of $A$ (counted with multiplicities) if the orders of $G$ and $A$ are coprime or if $A$ is nilpotent.}

Substantial progress has been achieved in the coprime case, that is the case where $(\left\vert G\right\vert ,\left\vert A\right\vert )=1$. This condition yields the following properties, which significantly facilitate the analysis of the structure of $G.$
\begin{enumerate}
\item \textit{$G=[G,A]C_{G}(A),$} 

\item \textit{$[G,A]=[G,A,A],$}

\item \textit{$%
C_{G/N}(A)=C_{G}(A)N/N,$ for any $A$-invariant normal subgroup $N$ of $G$}, 

\item \textit{there exists an $A$%
-invariant Sylow $p$-subgroup of $G$  for any prime $p$ dividing the order of $G$.}
\end{enumerate}
However, even in this case, the conjecture remains open. A seminal survey paper by Turull \cite{Turull1994}, who obtained the most definitive results concerning the coprime case, summarizes the state of research until 1994. His key findings can be summarized as follows.

\textbf{Theorem.} \textit{Let $A$ be a group acting on the solvable group $G$ and assume that $(|G|,|A|)=1$.
}
\begin{enumerate}

\item \textit{If $A$ is solvable, then $h(G)\leq 2\ell (A)+h(C_{G}(A))$}.

\item \textit{If A acts with regular orbits on G (that is, for all $B\leq A$ and for each $B$-invariant elementary abelian section $S$ of $G$, there exists $v\in S$ such that $C_B(S)=C_B(v)$), then $h(G)\leq \ell (A)+\ell (C_{G}(A))$.}

\end{enumerate}
Consequently, the conjecture is true in the coprime case if $A$ acts on $G$ with regular orbits. In the noncoprime case, aside from E. Dade's basic result \cite{Dade}, which provides an exponential bound for $h(G)$ in terms of $\ell(A)$, only scattered results are available. One of the most significant among these was obtained by Jabara in \cite{Jabara}, who established a quadratic bound, $h(G) \leq 7\ell(A)^2$, for the case where $A$ is cyclic.

In this paper, we make use of the numerical invariant $\mathbf{c}(G;A)$, introduced in \cite{nonco} and referred to as the fixing size of $A$ on $G$, which is associated with the pair $(G,A)$. The invariant $\mathbf{c}(G;A)$ appears to have the potential to play the role of 
$C_G(A)$-related invariants, particularly in the non-coprime setting. It is important to note that $\mathbf{c}(G;A)$ equals zero if the action is fixed point free and is equal to $\ell(C_{G}(A))$ in the coprime case. Furthermore, it possesses the desirable property that $$\mathbf{c}(G;A)=\mathbf{c}(G/N;A)+\mathbf{c}(N;A)$$for any $A$-invariant normal subgroup $N$ of $G$. This additivity resembles Property (3)  mentioned above in the coprime action case.

\textbf{We adopt the notation $\pi=\pi(G)\cap\pi(A)$ throughout the paper.}

Suppose the action of a nilpotent group $A$ on the solvable group $G$ is noncoprime. We first address the extreme case where the set $\pi$ is a singleton. For this case, we prove that $h(G)\leq 2\ell (A)$ provided the action of $A$ on $G$ is fixed point free. As is the case with the proof of main theorem of \cite{Turull1984}, this result is derived as a corollary to a more general result, namely
 we prove that $h(G)\leq 2\ell (A)+\mathbf{c}(G;A)$
under the additional assumption that a suitable Hall subgroup of $G$ is left invariant by $A.$ This additional requirement is admittedly an awkward assumption; it unnaturally relieves us of the inherent complexities of noncoprime action, effectively providing a partial analogue to Property (4) mentioned above of coprime actions. However, in the fixed point free case, this assumption is automatically satisfied, eliminating any additional constraint on our final objective. Namely, we prove that

\textbf{Theorem.} \textit{Let $G$ be a  solvable group, $A$ a nilpotent subgroup of $Aut(G)$, and let $|\pi|\leq 1$. If  $A$ leaves a Hall $\pi^{\prime}$-subgroup of $G$ invariant then 
 $$
h(G)\leq 2\ell (A)+\textbf{c}(G;A).$$}

We next study the case where $A$ acts with regular orbits. We derive the inequality $h(G)\leq \ell(A)$ when $A$ is of prime power order and is fixed point free as an immediate corollary of the following theorem.

\textbf{Theorem.} \textit{Let $G$ be a solvable group and $A$ be a group of order $p^n$ for some prime $p$, acting by automorphisms and with regular orbits on $G$. If $G$ contains an $A$-invariant Hall $p'$-subgroup of $G$ then $$h(G)\leq \ell(A)+\mathbf{c}(G;A).$$}

Additionally, we observe that the following holds:

\textbf{Theorem.} \textit{Let $G$ be a solvable group and $A$ be a nilpotent group acting with regular orbits on $G$. Suppose that $G$ has an $A$-invariant Hall $\pi'$-subgroup and that Hall $\pi$-subgroups of $G$ are abelian. Then $$h(G)\leq \ell(A)+\mathbf{c}(G;A)+2.$$ In particular, we have $h(G)\leq \ell(A)+2$ whenever $A $ is fixed point free and Hall $\pi$-subgroups of $G$ are abelian.}

In the last section we study the fixed point free action of a nilpotent group $A$ under the condition that Hall $\pi$-subgroups of $G$ are nilpotent and obtain the following.

\textbf{Theorem.} \textit{Let $G$ be a solvable group, and $A$ be a nilpotent subgroup of $Aut(G)$ acting fixed point freely on $G$.
Assume that Hall $\pi$-subgroups of $G$ are nilpotent. Then $h(G)\leq 2\ell (A).$}  
\section{The Fixing Size of A on G}
\begin{definition}
Let $A$ act on the group $G$. Let $$1=N_{k+1}\vartriangleleft
N_{k}\vartriangleleft N_{k-1}\vartriangleleft \cdots \vartriangleleft
N_{1}=G $$ be a series of $A$-invariant subnormal subgroups of $G$ such
that $N_{i}/N_{i+1}$ has no proper nontrivial $A$-invariant normal subgroup for each $i=1,\ldots ,k$,
in short an $A$-composition series of $G.$ We let $$\mathbf{c}(G;A) = \left\vert \left\{ i:[N_{i},A]\leq N_{i+1}\right\} \right\vert $$ and
call it \textbf{the fixing size of the action of }$A$\textbf{\ on }$G.$ This
number is an invariant of the pair $(G,A)$ and does not depend on the choice
of the $A$-composition series under consideration.
\end{definition}

\begin{proposition}\label{Prop 2} Let $A$ act on $G$. Then the following hold. 

\begin{enumerate} 
\item $\mathbf{c}(G;A)=\mathbf{c}(N;A)+\mathbf{c}(G/N;A)$ for every $A$-invariant normal subgroup $N$ of $G$.
\item If $1=N_{k+1}\vartriangleleft N_{k}\vartriangleleft N_{k-1}\vartriangleleft \cdots \vartriangleleft N_{1}=G$ is an $A$-invariant series of $G$ then $\mathbf{c}(G;A)=\sum_{i=1}^{k}\mathbf{c}(N_{i}/N_{i+1};A).$
\item $\mathbf{c}(H;A)\leq \mathbf{c}(G;A)$ for any $A$-invariant subgroup $H$ of $G$.
\end{enumerate}

\end{proposition}

\begin{proof}
This claim is obviously true.
\end{proof}

\begin{proposition}\label{Prop 3} Let $A$ act on $G$. Then we have the following.

$(1)$ If $G$ is a $p$-group and $A$ is $p$-nilpotent then $\mathbf{c}%
(G;A)=\ell (C_{G}(A_{p^{\prime }})).$

$(2)$ If $G$ is solvable and $(|G|,|A|)=1$ then $\mathbf{c}(G;A)=\ell(C_{G}(A)).$

\end{proposition}

\begin{proof}$(1)$ Let $1=N_{k+1}\vartriangleleft N_{k}\vartriangleleft
N_{k-1}\vartriangleleft \cdots \vartriangleleft N_{1}=G$ be an 
$A$-composition series of $G$. Suppose that $[N_i,A_{p'}]\leq N_{i+1}$. Then $C_{N_i/N_{i+1}}(A)=C_{N_i/N_{i+1}}(A_{p})$ is nontrivial and hence we obtain $[N_i,A]\leq N_{i+1}$. This means that $[N_i,A_{p'}]\leq N_{i+1}$ if and only if $[N_i,A]\leq N_{i+1}$. On the other hand notice that $$\ell(C_G(A_{p'})) = \sum_{i=1}^k \ell(C_{N_i/N_{i+1}}(A_{p'})).$$ Observe that $[N_i,A]\leq N_{i+1}$ implies $\ell(N_i/N_{i+1})=1$ and there is no contribution otherwise. Therefore we have $$\ell(C_G(A_{p'})) = |\{i : [N_i, A] \leq N_{i+1}\}| = c(G; A)$$ as claimed.

$(2)$ Let $G$ \ be a solvable group and $%
1=N_{k+1}\vartriangleleft N_{k}\vartriangleleft N_{k-1}\vartriangleleft
\cdots \vartriangleleft N_{1}=G$ be an $A$-composition series of $G.$  Since $G$ is solvable, each factor $N_i/N_{i+1}$ is elementary abelian. Clearly $%
\mathbf{c}(G;A)=\sum\nolimits_{i=1}^{k}\mathbf{c}(N_{i}/N_{i+1};A).$ As $%
(|G|,|A|)=1$ it holds that $A$ is $p$-nilpotent for any prime $p$ dividing
the order of $G$ with $O_{p^{\prime }}(A)=A$. It follows by (1) that $$\mathbf{c}%
(G;A)=\sum\nolimits_{i=1}^{k}\ell (C_{N_{i}/N_{i+1}}(A))=\ell (C_{G}(A).$$

\end{proof}

Henceforth, $H_{\sigma}$ will denote a Hall $\sigma$-subgroup of the group $H$ for a set of primes $\sigma.$

\begin{proposition}
\label{Prop 4}Let $A$ act on the solvable group $G$. Suppose that $A$
leaves invariant a Sylow subgroup $G_{p}$ and that $A$ is $p$-nilpotent for any $p\in \pi (G).$ Then 
\begin{equation*}
\mathbf{c}(G;A)=\sum\nolimits_{p\in \pi (G)}\ell (C_{G_{p}}(O_{p^{\prime
}}(A))).
\end{equation*}%
In particular, $\mathbf{c}(G;A)=0$ if and only if $C_{G}(A)=1.$
\end{proposition}

\begin{proof}
Set $\mathcal{S}$ be a set of $A$-invariant Sylow $p$-subgroups $p$ for all $p\in\pi(G).$ Pick a maximal $A$-invariant normal subgroup $N$ of $G$. Let $q$ the prime dividing $|G/N|$. Then $G_{p}\cap N$ is an $A$-invariant Sylow $p$-subgroup of $N$ for any $p\in \pi(N)$ with $p\ne q.$   It follows by induction that 
$$\mathbf{c}(N;A)=\ell (C_{G_{q}\cap N}(O_{q^{\prime
}}(A)))+\sum\nolimits_{q\neq p\in \pi (N)}\ell (C_{G_{p}}(O_{p^{\prime }}(A))).$$

On the other hand we set $Q=G_q$ and observe that $\mathbf{c}(G/N;A)=\mathbf{c}(Q/(Q\cap N);A)=\ell (C_{Q/(Q\cap
N}(O_{q^{\prime }}(A)))=\ell (C_{Q}(O_{q^{\prime }}(A))/C_{Q\cap
N}(O_{q^{\prime }}(A)))$.
This yields the result.
\end{proof}

\begin{proposition}
\label{Prop 5} Let $A$ act on the solvable group $G$. Suppose that $A$
leaves invariant a Sylow subgroup $G_{p}$ for any $p\in \pi (G)$ and that $A$ is nilpotent. Then

$(1)$ $\mathbf{c}(G;A)=\mathbf{c}(G_{p};A)+\mathbf{c}(G_{p^{\prime }};A),$

$(2)$ $\mathbf{c}(G_{p^{\prime }};A)=\mathbf{c}(C_{G_{p^{\prime
}}}(A_{p});A_{p^{\prime }})$ where $G_{p^{\prime }}=\prod\nolimits_{p\neq
q\in \pi (G)}G_{q}$.

\end{proposition}

\begin{proof}
$(1)$ is clear by the above proposition. Let $$1=H_{0}<H_{1}<\cdots <H_{n}=H$$ be an $A$-composition series of $
H=G_{p^{\prime }}$. Then letting $M_{i}=H_{i}/H_{i-1}$ we see that $$\mathbf{c
}(M_{i};A)=\mathbf{c}(C_{M_{i}}(A_{p});A)=\mathbf{c}(C_{M_{i}}(A_{p});A_{p^{%
\prime }}).$$ As $C_{M_{i}}(A_{p})$ is covered by $C_{H_{i}}(A_{p})$ we see
that $$\sum_{i=1}^{n}\mathbf{c}(C_{M_{i}}(A_{p});A_{p^{\prime }})=\mathbf{c}%
(C_{H}(A_{p});A_{p^{\prime }})=\sum_{i=1}^{n}\mathbf{c}(M_{i};A)=\mathbf{c}%
(H;A).$$
\end{proof}

\section{Bounding Fitting height in terms of fixing size}

Let $G$ be a $p$-solvable group. Define the characteristic subgroups $%
K_{j}=O_{p^{\prime },p,p^{\prime },p,p^{\prime },\ldots ,p^{\prime }}(G)$ and $%
L_{j}=O_{p^{\prime },p,p^{\prime },p,p^{\prime },\ldots ,p^{\prime },p}(G)$
where the number of $p$'s in the indices are $j-1$ and $j$ respectively $%
j=1,2,\ldots,m+1$. The normal series

\begin{equation*}
1\leq K_{1}<L_{1}<\cdots <K_{k}<L_{k}<\cdots <K_{m}<L_{m}\leq K_{m+1}=G
\end{equation*}%
is the upper $p$-series of the group $G$ with $p$-length $\ell _{p}(G)=m.$
We call $L_{m}/K_{m}$ the top $p$-factor and $L_{1}/K_{1}$ the bottom $p$%
-factor.

\begin{lem}\label{lem 1}
Let $G$ be a $p$-solvable group and $G_{p}$ a Sylow subgroup of $G$ and let%
\begin{equation*}
1\leq K_{1}<L_{1}<\cdots <K_{k}<L_{k}<\cdots <K_{m}<L_{m}\leq K_{m+1}=G
\end{equation*}%
be its upper $p$-series. Then $\ell ((L_{m-j}/K_{m-j})/\Phi
(L_{m-j}/K_{m-j})\geq j+1$ for any $j=0,1,\ldots ,m-1$ and hence 
\begin{equation*}
2\ell _{p}(G)-1\leq \frac{\ell _{p}(G)(\ell _{p}(G)+1)}{2}\leq \ell (G_{p}).
\end{equation*}
Furthermore  in the case of equality  the top $p$-factor
is of order $p$ and we have $m=1$ or $p=2=m$ and $G=L_{m}.$
\end{lem}

\begin{proof}
Clearly $\ell (L_{m}/K_{m})/\Phi  (L_{m}/K_{m})\geq 1.$ Suppose there exists $i\in \mathbb{N}$ such that $\ell ((L_{m-i}/K_{m-i})/\Phi  (L_{m-i}/K_{m-i}))<i+1$ and let $j$ be the smallest such
integer then the rank $n$ of the Frattini factor group of $L_{m-j}/K_{m-j}$
satisfies $n\leq j$. As $G/L_{m-j}$ is faithfully represented on this
Frattini factor group we see that $G/L_{m-j}$ is isomorphic to a subgroup of 
$GL(n,p).$ A Sylow p-subgroup of $GL(n,p)$ has order $p^{%
\frac{(n-1)n}{2}}.$ Notice that $\frac{(n-1)n}{2}\leq \frac{(j-1)j}{2}$, but the order of
a Sylow $p$-subgroup of $G/L_{m-j}$ is $p^{k}$ with $$k=\sum_{r=0}^{j-1}\ell (L_{m-r}/K_{m-r})\geq
\sum_{r=0}^{j-1}(r+1)=\frac{j(j+1)}{2}$$ proving the first
claim.

Next suppose that $\ell _{p}(G)(\ell _{p}(G)+1)=2\ell (G_{p}).$ Then it holds 
that $$\ell (L_{m-i}/K_{m-i})=i+1=\ell ((L_{m-i}/K_{m-i})/\Phi
(L_{m-i}/K_{m-i})$$ that is, $L_{m-i}/K_{m-i}$ is an elementary
abelian $p$-group of rank $i+1$ for any $i=0,1,\ldots ,m-1.$ If $m>1$ then $
L_{m-1}/K_{m-1}$ is elementary abelian of order $p^{2}$. Therefore $
G/L_{m-1}$ is isomorphic to a subgroup $H$ of $GL(2,p)$ such that $
O_{p^{\prime }}(H)$ is nontrivial and a Sylow subgroup of $H$ acts
nontrivially on $O_{p^{\prime }}(H).$ By Lemma \ref{lem 2} below we see that either $m=1$ or $p=2$ and $G/L_{m-1}$ is isomorphic to $GL(2,2)$. If $m>2$ consider the group $G/L_{m-2}$ which is faithfully represented on the 3-dimensional vector space $L_{m-2}/K_{m-2}$. Thus $G/L_{m-2}$ is a solvable subgroup of $GL(3,2)$ having a proper quotient $G/K_{m-1}$ isomorphic to $S_{4}$. This contradiction shows that $p=2=m$ and $G=L_{m}$.
\end{proof}
 \begin{lem} \label{lem 2} Let $H$ be a subgroup of $GL(2,p)$ such that $1\neq K=O_{p^{\prime }}(H)$ with  $C_{H}(K)\leq K$. Then $K=H$ or $p=2$.
\end{lem}
\begin{proof}
   We have $|GL(2,p)|=p(p^{2}-1)(p-1)$ and $|K|$ divides $(p^{2}-1)(p-1)$. Assume that $H\ne K.$ Then $p$ divides the order of $H.$ Let $P$ be a Sylow $p$-subgroup of $H$. As $[P,K]\neq 1$ by hypothesis we may choose $r\in \pi(K)$ and  a Sylow $r$-subgroup $R$ of $K$ such that $[P,R]\neq 1$. By the Frattini argument we may assume that $P\leq N_{H}(R)$.  It also holds that  $r$ divides $p^2-1$ and hence $r$ divides $p+1$, because $[P,R]\neq 1$. Thus we have $r=p+1$ and so $p=2$, that is, $GL(2,p)$ is isomorphic to $S_{3}.$ 
\end{proof}

\begin{theo}\label{theo 3}
Let $G$ be a solvable group, $A$ a nilpotent subgroup of $Aut(G)$. Assume that $|\pi|\leq 1$ and that $G$
contains an $A$-invariant Hall $p^{\prime }$-subgroup $G_{p^{\prime }}$ if $
p\in \pi .$ Then%
\begin{equation*}
h(G)\leq 2\ell (A)+\mathbf{c}(G;A).
\end{equation*}
\end{theo}
\begin{proof}
Let $(G,A)$ be a counterexample to the theorem with minimal $\left\vert
G\right\vert +\left\vert A\right\vert $. Suppose first that $[G,A]<G.$ Then we have $$h([G;A])\leq 2\ell(A)+\mathbf{c}([G;A];A) =2\ell(A)+\mathbf{c}(G;A) -\mathbf{c}(G/[G;A];A)$$ $$=2\ell(A)+\mathbf{c}(G;A)-\ell (G/[G;A])$$ and hence $$h(G)\leq h([G;A])+\ell(G/[G;A])\leq 2\ell(A)+\mathbf{c}(G;A)$$ which is impossible. It follows that $G=[G,A].$

If $\pi (G)\cap \pi (A)=\emptyset $
then by  \cite{Turull1984} we have $h(G)\leq 2\ell (A)+h(C_{G}(A)).$ But $%
\mathbf{c}(G;A)=\ell (C_{G}(A))\geq h(C_{G}(A))$ by Proposition \ref{Prop 3}.  Thus we can assume that $%
\pi (G)\cap \pi (A)=\{p\}.$ Let us write $A=A_{p}\times A_{p^{\prime }}$ and
observe that $A_{p^{\prime }}$ acts coprimely on $G.$ Thus we get again by %
\cite{Turull1984}  that $h(G)\leq 2\ell (A_{p^{\prime }})+h(C_{G}(A_{p^{\prime }})).
$ If $A_{p^{\prime }}\neq 1$ we see that $C_{G}(A_{p^{\prime }})<$ $G$. So $%
(C_{G}(A_{p^{\prime }}),A)$ is not a counterexample to the theorem. So we
get 
\begin{equation*}
h(C_{G}(A_{p^{\prime }}))\leq 2\ell (A/C_{A}(C_{G}(A_{p^{\prime }})))+%
\mathbf{c}(C_{G}(A_{p^{\prime }});A)
\end{equation*}
giving by Proposition \ref{Prop 2} that 
\begin{equation*}
h(G)\leq 2\ell (A_{p^{\prime }})+2\ell (A/C_{A}(C_{G}(A_{p^{\prime }})))+%
\mathbf{c}(C_{G}(A_{p^{\prime }});A)\leq 2\ell (A)+\mathbf{c}(G;A).
\end{equation*} 

Hence we may assume that $A$ is a group of order $p^{n}$ for some $n=\ell (A)$
where $p$ is a divisor of the order of $G$.

By assumption there exists a Hall $p^{\prime }$-subgroup  $S$ of $G$ which is $A$%
-invariant. Let $P\in Syl_{p}(G).$ By \cite{Casolo} Lemma 4.1
applied to the factorization $G=PS$  we see that 
\begin{equation*}
h(G)\leq h(S)+2\ell _{p}(G).
\end{equation*}
Thus we get  $2\ell (A)+\mathbf{c}(G;A)+1\leq h(G)\leq 2\ell _{p}(G)+h(S).$
On the other hand we know by \cite{Turull1984} $\ $that $h(S)\leq 2\ell
(A)+h(C_{S}(A))$. Therefore $$2\ell (A)+\mathbf{c}(G;A)+1\leq 2\ell
_{p}(G)+2\ell (A)+h(C_{S}(A))$$ and as $\mathbf{c}(G;A)=\ell (C_{S}(A))+\ell
(P)$ by Proposition \ref{Prop 5}(1) and\ $2\ell _{p}(G)\leq \ell (P)+1$ by Lemma \ref{lem 1} we get%
\begin{equation*}
c(G;A)+1\leq 2\ell _{p}(G)+h(C_{S}(A))\leq \ell (P)+1+\ell (C_{S}(A))=%
\mathbf{c}(G;A)+1.
\end{equation*}%
and hence $\ell (C_{S}(A))=h(C_{S}(A))$ and $2\ell _{p}(G)=\ell (P)+1$. Then $\ell _{p}(G)=1$ and $P$ and hence the top $p$-factor $L_{1}/K_{1}$
of the upper $p$-series of $G$ is cyclic of prime order.  So $[L_{1},A]\leq
K_{1}$ and hence by the three subgroup lemma $[[G,A],L_{1}]\leq K_{1}$ that
is $G=[G,A]\leq L_{1}.$ But if $G=L_{1},$ then we have $[G,A]\leq K_{1}$
giving the contradiction $G=[G,A]\leq K_{1}<G.$
\end{proof}

An argument similar to the one in Theorem \ref{theo 3} gives a slightly better bound in the special case that $A$ acts with regular orbits:

\begin{theo} \label{theo 6} Let $G$ be a solvable group and $A$ be a group of prime power of order $p^n$ acting by automorphisms and with regular orbits on $G$. If $G$ contains an $A$-invariant Hall $p'$-subgroup of $G$ then $$h(G)\leq \ell(A)+\mathbf{c}(G;A).$$
\end{theo}

\begin{proof}
    Let $(G,A)$ constitute a counterexample to the theorem with minimal $|G|+|A|$. Let $S$ be an $A$-invariant Hall $p'$-subgroup of $G.$ Then we have 
    $$h(G)\leq h(S)+2\ell _{p}(G)$$
    and by Theorem A in \cite{Turullgac} $$
    h(S)\leq \ell (A)+\ell(C_{S}(A)).
   $$
    If we have $2\ell _{p}(G)\leq \ell (G_{p})$ then the above observations imply by Lemma \ref{lem 1} that $$h(G)\leq h(S)+2\ell _{p}(G)\leq \ell (A)+\ell(C_{S}(A))+\ell (G_{p})=\ell (A)+\mathbf{c}(G;A).$$ Thus we have $\ell (G_{p}) < 2\ell _{p}(G)$.
   But by Lemma \ref{lem 1}  we also know that $$\ell _{p}(G)(\ell _{p}(G)+1)\leq 2\ell (G_{p})< 4\ell _{p}(G)$$ giving  $\ell _{p}(G) < 3$. So we have either $\ell _{p}(G) = 2$ and $ \ell (G_{p})\leq 3$ or $\ell _{p}(G) = 1= \ell (G_{p})$. In both cases top $p$-factor of $G$ is of prime order and hence is centralized by $A.$ This gives a contradiction as in the proof of Theorem \ref{theo 3}.
\end{proof}

\begin{theo} \label {theo 7} Let $G$ be a solvable group and $A$ be a nilpotent group acting with regular orbits on $G$. Suppose that $G$ has an $A$-invariant Hall $\pi'$-subgroup and that Hall $\pi$-subgroups of $G$ are nilpotent then $$h(G)\leq \ell(A)+\mathbf{c}(G_{{\pi'}};A)+2d(G_{\pi})=\ell(A)+\ell(C_{G_{{\pi'}}}(A))+2d(G_{\pi}).$$ In particular, we have $h(G)\leq \ell(A)+\mathbf{c}(G_{{\pi'}};A)+2$ when Hall $\pi$-subgroups of $G$ are abelian.
    
\end{theo}
\begin{proof} By Lemma 4.1 in \cite{Casolo}, $h(G)\leq h(G_{\pi})+ h(G_{{\pi'}}) +2\ell_{\pi}(G)-1.$ On the other hand Theorem 2.3 in \cite{Casolo} gives that $\ell_{\pi}(G)\leq d(G_{\pi})$. Then we get $$h(G)\leq  h(G_{{\pi'}}) +2d(G_{\pi}).$$ Theorem A in \cite{Turullgac} applied to the coprime action of $A$ on the group $G_{\pi'}$ yields that $h(G_{{\pi'}})\leq \ell(A)+\mathbf{c}(G_{{\pi'}};A).$ Thus we have $$h(G)\leq  \ell(A) +\mathbf{c}(G_{{\pi'}};A)+2d(G_{\pi}).$$ 
\end{proof}
\section{Fixed point free case}

Let $A\leq Aut(G)$. Notice that $\mathbf{c}(G;A)=0$ if and only if $C_{G}(A)=1.$ In this section we study the fixed point free action of a nilpotent group $A$ under the additional condition that Hall $\pi$-subgroups of $G$ are nilpotent as a generalization of the case where $|\pi|\leq 1$.

 \begin{definition}An $A$-tower of length $t$ in $G$ is a sequence of $A$-invariant subgroups $S_i, \, i=1,\ldots ,t,$ of $G$  such that the following hold:

		(\emph{a}) $S_i$ is a nontrivial $p_i$-group for some prime $p_i$ for $i = 1, \ldots, t$,
        
        (\emph{b}) $S_i$ normalizes $S_j$ for $i<j$, 
  
		(\emph{c}) Set $P_t=S_t$ and $P_i=S_i/C_{S_i}(P_{i+1})$ for $i = 1, \ldots, t-1$. $P_i$ is nontrivial for $i = 1, \ldots, t$, 
        
        (\emph{d}) $p_i \neq p_{i+1}$, for $i = 1, \ldots, h-1$.
        \end{definition}
\begin{lem}	Let $A$ act on the solvable group $G$. Suppose that every $A$-invariant subgroup of $G$ has an $A$-invariant Sylow subgroup for any prime dividing its order. Let $N$ be an $A$-invariant normal subgroup of $G$ and let $\overline{S_t},\ldots ,\overline{S_1}$ be an $A$-tower of $G/N$. Then 
there is an $A$-tower $S_t,\ldots ,S_1$ of $G$ which maps to $\overline{S_i}$ for $i =1,\ldots t$. 
\end{lem}
\begin{proof}  Let $G$ be a counterexample with $|G|+t$ minimal. 
Let $\overline{S_i}=H_i/N$. By hypothesis we can choose an $A$-invariant Sylow $p_t$-subgroup $S_t$ of $H_t$. Notice that for each $i=1,\ldots ,t$, $$\overline{S_i}=N_{\overline{S_i}}(\overline{S_t})=N_{H_i}(H_t)N/N=N_{H_i}(S_t)N/N$$ by the Frattini argument which yields that  $H_i=N_{H_i}(S_t)N.$  
Clearly, the sequence $$N_{H_{t-1}}(S_{t})N/N,\ldots , N_{H_1}(S_{t})N/N$$ forms an $A$-tower of $N_{G}(S_t)N/N.$ Then by induction it extends to an $A$-tower of $N_G(S_t)$ and hence an $A$-tower of $G$ of length $t-1$. Adjoining $S_t$ to the resulting tower we obtain an $A$-tower of $G$ of length $t$, which is impossible. Hence the claim is established. 
\end{proof}
\begin{lem} Let $A$ act on the solvable group $G$. Suppose that every $A$-invariant subgroup of $G$ has an $A$-invariant Sylow subgroup for any prime dividing its order. Then $G$ contains an $A$-tower of length $h=h(G).$
\end{lem}
\begin{proof} Let $(G,A)$ be a counterexample with $|G|$ minimal. Set $\bar{G}=G/F(G)$. Set $h=h(G)$. By induction there is an $A$-tower $U_{h-1},\ldots , U_{1}$ of $\bar{G}$. By the above lemma there is an $A$-tower $S_{h-1},\ldots ,S_1$ of $G$ which maps to $U_{h-1},\ldots ,U_1$. Clearly there exists a prime $p$ dividing the order of $G$ such that $[O_p(G), S_{h-1}]\ne 1$. Set $S_{h}=O_p(G).$ Then the tower is $S_{h},S_{h-1},\ldots ,S_1$ is a desired one.
\end{proof}
\begin{theo}\label{theo 4.1} Let $G$ be a solvable group, and $A$ be a nilpotent subgroup of $Aut(G)$ acting fixed point freely on $G$. Assume that Hall $\pi$-subgroups of $G$ are nilpotent where $\pi=\pi(G)\cap \pi(A)$. Then $$h(G)\leq 2\ell (A)$$.
\end{theo}
\begin{proof} Let $(G,A)$ constitute a counterexample to the theorem with minimal $|G|+|A|$. As the nilpotent group $A$ acts fixed point freely on $G$, every $A$-invariant subgroup of $G$ contains an $A$-invariant Sylow subgroup for each prime by Lemma 8.1 in \cite{Dade}. This leads to the existence of an $A$-tower $S_i, \, i=1,\ldots ,h=h(G),$ in $G$ by the above lemma. It follows by induction that $G=\Pi_{i=1}^h S_i.$ Since $G_{\pi}$ is nilpotent, the set $\pi$ must be a singleton, that is, there is only one prime dividing $(|G|,|A|)$. Hence the result follows by Theorem \ref{theo 3}. 
\end{proof}

\end{document}